\documentclass[12pt]{amsart}
\usepackage{amsmath,amsthm,enumerate,graphics,amsfonts,latexsym,amsopn,verbatim,amscd,amssymb}
\usepackage{color}

\theoremstyle{plain}
\newtheorem{thm}{Theorem}[section]

\newtheorem{lem}[thm]{Lemma}
\newtheorem{cor}[thm]{Corollary}
\newtheorem{prop}[thm]{Proposition}
\newtheorem{conj}[thm]{Conjecture}
\newtheorem{ques}[thm]{Question}
\newtheorem{prob}[thm]{Problem}

\theoremstyle{definition}
\newtheorem{rem}[thm]{Remark}

\DeclareMathOperator{\Sol}{Sol}
\DeclareMathOperator{\nbd}{nbd}

\begin{document}

\title[Non-solvable graphs of  groups]{Non-solvable graphs of  groups}

\author[P. Bhowal, R. K. Nath  and  D. Nongsiang]{P. Bhowal, R. K. Nath  and  D. Nongsiang}
\address{Parthajit Bhowal, Department of Mathematical Science, Tezpur  University, Napaam-784028, Sonitpur, Assam, India.}

\email{bhowal.parthajit8@gmail.com}


\address{Rajat Kanti Nath,  Department of Mathematical Science, Tezpur  University, Napaam-784028, Sonitpur, Assam, India.}

\email{rajatkantinath@yahoo.com}

\address{Deiborlang Nongsiang, Department of Mathematics, Union Christian College, Umiam-793122, Meghalaya, India.}

\email{ndeiborlang@yahoo.in}

\begin{abstract} 
Let $G$ be a group and $\Sol(G)=\{x \in G : \langle x,y \rangle \text{ is solvable for all } y \in G\}$. We associate a graph $\mathcal{NS}_G$ (called the non-solvable graph of $G$) with $G$ whose vertex set is  $G \setminus \Sol(G)$ and two distinct vertices are adjacent if they generate a non-solvable subgroup. In this paper we study many properties of $\mathcal{NS}_G$. In particular, we obtain results on vertex degree, cardinality of vertex degree set, graph realization, 
domination number, vertex connectivity, independence number and clique number of $\mathcal{NS}_G$. We also  consider two groups $G$ and $H$ having isomorphic non-solvable graphs and derive some properties of $G$ and $H$.  Finally, we conclude this paper by showing that $\mathcal{NS}_G$ is  neither planar, toroidal, double-toroidal, triple-toroidal  nor projective. 

  
\end{abstract}


\maketitle

\section{Introduction} \label{S:intro} 

Let $G$ be a group (finite/infinite) and $\Sol(G)=\{x \in G : \langle x,y \rangle \text{ is solvable for all } y \in G\}$. Note that $\Sol(G)= \underset{x \in G}{\cap} \Sol_G(x)$ where $\Sol_G(x) = \{g \in G : \langle g,x \rangle \text{ is solvable}\}$ is called solvabilizer of $x$ in $G$. In general, $\Sol_G(x)$ is not a subgroup of $G$. However, it was shown,  in \cite{gkps}, that $\Sol(G)$ is the solvable radical of $G$ if $G$ is a finite group. That is, if $G$ is a finite group then $\Sol(G)$ is the unique largest solvable normal subgroup of $G$.
  We associate a simple graph $\mathcal{NS}_G$ with $G$, called the  non-solvable graph of $G$ whose vertex set is   $G \setminus \Sol(G)$ and two distinct vertices  $x$ and $y$ are adjacent if and only if $\langle x,y \rangle$ is not solvable. The non-solvable graph of $G$ was introduced in \cite{hr} and studied in \cite{hr, akbari}. The complement of $\mathcal{NS}_G$, known as solvable graph of $G$, is considered in \cite{bnn19} recently. The concept of non-solvable graph of a group is an extension of non-nilpotent graph and hence non-commuting graph of groups. The non-nilpotent and non-commuting graph of finite groups are studied extensively in \cite{Ab10,ns2017} and \cite{Ab06, Af15,Dar14,ddn18,talebi} respectively.

We write $V(\Gamma)$ to denote the vertex set of a graph $\Gamma$. The degree of a vertex $x \in V(\Gamma)$ denoted by $\deg(x)$ is defined to be the number of vertices adjacent to $x$ and $\deg(\Gamma) = \{\deg(x) : x \in V(\Gamma)\}$ is the vertex degree set of $\Gamma$. In Section 2, we shall study some properties of  degree of a vertex and vertex degree set of $\mathcal{NS}_G$. In this section, we also obtain some bounds for solvability degree of a finite group and one of which is better than an existing bound, in particular \cite[Theorem 4.3]{bnn19}.
The solvability degree of a finite group $G$ is the probability that a randomly chosen pair of elements of $G$ generate a solvable group (see \cite{gW2000,bnn19,wIL2008}). 

 It is known that $\mathcal{NS}_G$ is neither a tree nor a complete graph (see \cite{hr, akbari}). In Section 3, we shall show that $\mathcal{NS}_G$ is not bipartite, more generally it is not complete multi-partite. We shall also show that $\mathcal{NS}_G$ is hamiltonian for some classes of finite groups. In Sections 4-6, we shall obtain several results regarding domination number, vertex connectivity, independence number and clique number of $\mathcal{NS}_G$. In section 7, we shall consider two groups $G$ and $H$ having isomorphic non-solvable graphs and derive some properties of $G$ and $H$. It was shown in \cite{hr} that $\mathcal{NS}_G$ is not planar. In the last section, we shall 
 show that the genus of  $\mathcal{NS}_G$ is greater or equal to $4$. Hence, $\mathcal{NS}_G$ is neither planar, toroidal, double-toroidal nor triple-toroidal. We conclude this paper by showing that $\mathcal{NS}_G$ is not projective.

Let $U$ be a nonempty subset of the vertex set of a graph $\Gamma$. The \textit{induced subgraph} of $\Gamma$ on $U$  is defined to be the graph $\Gamma [U]$  in which the vertex set  is $U$ and  the edge set consists precisely  those edges in $\Gamma$ whose endpoints lie in $U$.  For any non-empty subset $S$ of $V(\Gamma)$, we also write $\Gamma \setminus S$ to denote $\Gamma [V(\Gamma)\setminus S]$.

\section{Vertex degree and Cardinality of vertex degree set}
The {\it neighborhood} of a vertex $x$ in a graph $\Gamma$, denoted by $\nbd(x)$, is defined to be the set of all vertices adjacent to $x$ 
and so $\deg(x) =|\nbd(x)|$. It is easy to see that $\deg(x) = |G| - |\Sol_G(x)|$ for any vertex $x$ in the non-solvable graph ${\mathcal{NS}}_G$ of the group $G$.    
 In \cite{hr}, Hai-Reuven have shown that
\begin{equation}\label{1}
6 \leq \deg(x) \leq |G| - |\Sol(G)| - 2
\end{equation}
for any $x \in G \setminus \Sol(G)$. In this section, we first obtain some bounds for $P_s(G)$ using \eqref{1}. Recall that $P_s(G)$ is the solvability degree  of a finite group $G$ defined by the  ratio
\[ 
P_s(G) = \frac{|\{(u, v) \in G \times G : \langle u, v\rangle \mbox{ is solvable}\}|}{|G|^2}.
\]
It is worth mentioning that several properties of $P_s(G)$ including some bounds are studied in \cite{gW2000,bnn19,wIL2008}. The following result gives a connection between $P_s(G)$ and the number of edges in $\mathcal{NS}_G$.

\begin{lem}\label{degs}
If $G$ is a finite non-solvable group  then
\[ 
\sum_{x\in G\setminus \Sol(G)} \deg(x)= |G|^2(1-P_S(G)).
\]
\end{lem}
\begin{proof}
Let $U=\{(x,y)\in G\times G : \langle x,y \rangle \mbox{ is not solvable}\}$. Then
\begin{equation*}
|U|= |G\times G|- |\{(x,y)\in G\times G : \langle x,y \rangle \mbox{ is solvable}\}|= |G|^2-P_S(G)|G|^2.
\end{equation*}
Note that
\[
|U| = 2\times \mbox{number of edges in }\mathcal{NS}_G=\sum_{x\in G\setminus \Sol(G)} \deg(x).
\]
Hence the result follows.
\end{proof}

Now we obtain the following bounds for $P_S(G)$.
\begin{thm}\label{new-bounds}
If $G$ is a finite non-solvable group then
\[
\frac{2(|G|-|\Sol(G)|)}{|G|^2} + \frac{2|\Sol(G)|}{|G|} - \frac{|\Sol(G)|^2}{|G|^2} \leq P_S(G) \leq 1 - \frac{6(|G|-|\Sol(G)|)}{|G|^2}.
\]
\end{thm}

\begin{proof}
By Lemma \ref{degs} and \eqref{1}, we have 
\[
6(|G|-|\Sol(G)|)\leq |G|^2(1-P_S(G)) \leq (|G|-|\Sol(G)|)(|G|-|\Sol(G)|-2)
\]
and hence the result follows on simplification.
\end{proof}

\noindent It was shown in \cite[Theorem 4.3]{bnn19}   that 
\begin{equation}\label{2}
P_S(G) \geq \frac{2(|G|-|\Sol(G)|)}{|G|^2} + \frac{|\Sol(G)|}{|G|}.
\end{equation}
Note that  $\frac{|\Sol(G)|}{|G|} - \frac{|\Sol(G)|^2}{|G|^2} > 0$ for any finite non-solvable group $G$. Hence, the lower bound obtained in Theorem \ref{new-bounds} for $P_S(G)$ is better  than \eqref{2}.

   It was also shown, in \cite{hr}, that  $|\deg(\mathcal{NS}_G)| \ne 2$, where $\deg(\mathcal{NS}_G)$ is the vertex degree set of $\mathcal{NS}_G$. However, we observe that the cardinality of  $\deg(\mathcal{NS}_G)$ may be equal to $3$. In this section, we shall obtain a class of groups $G$ such that  $|\deg(\mathcal{NS}_G)| = 3$. Note that  $\deg(\mathcal{NS}_{A_5}) = \{24, 36, 50\}$.  More generally, we have the following result.  

\begin{prop}\label{degree-lem-1} 
Let $S$ be any finite solvable group. Then  $|\deg(\mathcal{NS}_{A_5\times S})| = 3$.
\end{prop}
The proof of Proposition \ref{degree-lem-1} follows from the fact that 
$|\deg(\mathcal{NS}_{A_5})| = 3$  and  the result given below.
 
\begin{lem}
 Let $G$ be a finite non-solvable group and $S$ be any finite solvable group. Then  
$|\deg(\mathcal{NS}_G)| =  |\deg(\mathcal{NS}_{G \times S})|$. 
\end{lem}
\begin{proof}
Let $(x,s),(y,t)\in G\times S$, then $\langle (x,s),(y,t)\rangle \subseteq \langle x,y\rangle \times \langle s,t\rangle$. Therefore, $\langle (x,s),(y,t)\rangle$ is solvable if and only if $\langle x, y\rangle$ is solvable. Also, $\Sol_{G\times S}((x,s)) = \Sol_G(x)\times S$ and hence $\nbd((x, s)) = \nbd(x)\times S$. That is, $\deg((x,s)) = |S|\deg(x)$. This completes the proof.
\end{proof}
 
 Now we state the main result of this section.

\begin{thm} \label{degree-main-Th}
If $G$ is a finite non-solvable group such that $G/\Sol(G)\cong A_5$ then $|\deg(\mathcal{NS}_G)| = 3$. 
\end{thm}
To prove this theorem   we need the following results.
\begin{lem} \label{sol} 
Let $H$ be a subgroup of a finite group $G$ and $x, y \in G$.
\begin{enumerate}
\item[\rm (a)] If $H$ is solvable then $\langle H,\Sol(G)\rangle$ is also solvable subgroup of $G$.

\item[\rm (b)] If $\langle x, y \rangle$ is solvable then $\langle xu, yv \rangle$ is also solvable for all $u, v \in \Sol(G)$.

\item[\rm (c)] If $\langle x,y \rangle$ is not solvable then  $\langle xu, yv \rangle$ is not solvable for all $u, v \in \Sol(G)$.
\end{enumerate}
\end{lem}
\begin{proof}
Part (a) is proved in \cite[Lemma 3.1]{bnn19}.   Part (b) follows from part (a). Also, note that parts (b) and (c) are equivalent.
\end{proof}

\begin{lem}\label{solo} 
Let $G$ be a finite group and $x,y \in G$. Then $\langle x\Sol(G),y\Sol(G) \rangle$ is solvable if and only if $\langle x, y \rangle$ is solvable.
\end{lem}
\begin{proof}
Let $H = \langle x, y \rangle$ and $Z = \Sol(G)$. Note that $\langle xZ, yZ \rangle = \frac{HZ}{Z}$. 
%
Suppose $\langle xZ, yZ \rangle$ is solvable. Then $\frac{HZ}{Z}$ is solvable. Since $Z\subset \Sol(HZ)$ and $Z$ is a normal subgroup of $HZ$, by \cite[Lemma 2.11 (3)]{hr}, we have 
\[
\frac{\Sol_{HZ}(x)}{Z} = \Sol_{\frac{HZ}{Z}}(xZ)= \frac{HZ}{Z}.
\]
Therefore, $\Sol_{HZ}(x)=HZ$. In particular, $\Sol_{H}(x)=H$ and so $H$ is solvable.

If $H$ is solvable then, by Lemma \ref{sol} (b), $\Sol_{HZ}(x)=HZ$ for all $x\in HZ$. Thus $HZ$ is solvable and so $\frac{HZ}{Z}$ is solvable. Hence, $\langle x_iZ,x_jZ \rangle$ is solvable for $x_i, x_j \in HZ$ and so  $\langle xZ, yZ \rangle$ is solvable.
\end{proof}

\begin{prop} \label{degsol}
 Let $G$ be a finite non-solvable group. Then  for all $x \in G \setminus \Sol(G)$ we have 
 \[
 \deg(x)=\deg(x\Sol(G))|\Sol(G)|.
 \]
\end{prop}
\begin{proof}
 Let $y \in \nbd(x)$. By Lemma \ref{sol}, we have $yz\in \nbd(x)$ for all $z \in \Sol(G)$. Thus $\nbd(x)$ is a union of distinct cosets of $\Sol(G)$. Let $\nbd(x)= y_1\Sol(G)\cup y_2\Sol(G)\cup \dots \cup y_n\Sol(G)$.  Then $\deg(x) = n|\Sol(G)|$. By Lemma \ref{solo}, we have $\langle x\Sol(G),y_i\Sol(G) \rangle$ is not solvable if and only if $\langle x,y_i \rangle$ is not solvable. Therefore, $\nbd(x\Sol(G)) = \{y_1\Sol(G), y_2\Sol(G),  \dots,  y_n\Sol(G)\}$ in $\mathcal{NS}_{G/\Sol(G)}$. Hence, $\deg(x\Sol(G)) = n$ and the result follows.
\end{proof}

As a consequence of Proposition \ref{degsol} we have the following corollary.
\begin{cor} \label{degsol-cor}
 Let $G$ be a finite non-solvable group. Then $|\deg(\mathcal{NS}_{G/\Sol(G)})| = |\deg(\mathcal{NS}_G)|$.
\end{cor}

\noindent \textbf{Proof of Theorem \ref{degree-main-Th}:}
Note that $G/\Sol(G)\cong A_5$ implies ${\mathcal{NS}}_{G/\Sol(G)} \cong {\mathcal{NS}}_{A_5}$. Therefore
\[
\left|\deg\left(\mathcal{NS}_{G/\Sol(G)}\right)\right| = |\deg(\mathcal{NS}_{A_5})| = 3.
\] 
Hence, the result follows from Corollary \ref{degsol-cor}.

We conclude this section with the following upper bound for $|\deg(\mathcal{NS}_G)|$. 
\begin{thm}
If $G$ is a finite non-solvable group having $n$ distinct solvabilizers then $|\deg(\mathcal{NS}_G)| \leq n - 1$.
\end{thm}
\begin{proof}
Let $G, X_1, X_2, \dots, X_{n - 1}$ be the distinct solvabilizers  of $G$ where $\Sol_G(x_i) = X_i$ for some $x_i \in G \setminus \Sol(G)$ and $i = 1, 2, \dots, n - 1$. Then 
\[
\deg(\mathcal{NS}_G) = \{|G| - |X_1|, |G| - |X_2|, \dots, |G| - |X_{n - 1}|\}.
\]
 Hence, the result follows. 
\end{proof}

\section{Graph Realization}
In \cite{hr}, it was shown that $\mathcal{NS}_G$ is connected with diameter two. It is also shown that  $\mathcal{NS}_G$ is not regular and hence  not a complete graph. Recently, Akbari \cite{akbari} have shown that $\mathcal{NS}_G$ is not a tree. In this section, we shall show that $\mathcal{NS}_G$ is not a complete multi-partite graph. We shall also  show that $\mathcal{NS}_G$ is hamiltonian for some groups. The following results are useful in this regard.

\begin{lem}\cite{hr}\label{solv} 
Let $G$ be a finite group. Then G is solvable if and only if $\Sol_G(x)$ is a subgroup of $G$ for all $x\in G$.
\end{lem}

\begin{thm}\cite{gkps}\label{thomson} 
 Let $G$ be a finite non-solvable group and $x, y \in G \setminus \Sol(G)$. Then  there exists $s \in G \setminus \Sol(G)$ such that $\langle x, s \rangle$ and $\langle y, s \rangle$ are not solvable.
\end{thm}





\begin{thm}\label{prop-nn}
 Let $G$ be a finite non-solvable group. Then $\mathcal{NS}_G$ is not a complete multi-partite graph. In particular, $\mathcal{NS}_G$ is not a complete bipartite graph.
\end{thm}
\begin{proof}
 Suppose $\mathcal{NS}_G$ is a complete multi-partite graph. Let $X_1, X_2, \dots, X_n$ be the partite sets. Let $x \in G \setminus \Sol(G)$, then $x \in X_i$ for some $i$ and $\Sol_G(x) = \Sol(G) \cup X_i$. Let $y, z \in \Sol_G(x)$. Then $\langle y,z \rangle$ is solvable and $yz \in \Sol_G(y)=\Sol_G(x)$. Thus $\Sol_G(x)$ is a subgroup of $G$. By Lemma \ref{solv}, $G$ is solvable, a contradiction. Hence, the result follows.
\end{proof}

\begin{thm} 
Let $G$ be a finite non-solvable group. Then $\mathcal{NS}_G$ is not a bipartite graph.
\end{thm}
\begin{proof}
Suppose $\mathcal{NS}_G$ is a bipartite graph. Let $X, Y$ be the partite sets.  Let $x \in X$ and $y \in Y$. Then, by Theorem \ref{thomson}, there exists $z \in G\setminus \Sol(G)$ such that $\langle x, z \rangle$  and   $\langle y, z \rangle$ are not solvable. Therefore, $z \notin X \cup Y$, a contradiction. Hence the result follows.
\end{proof}

\begin{thm} \label{hamiltonian-3}
Let $G$ be a finite non-solvable group such that $|\Sol_G(x)|\leq \frac{|G|}{2}$ for all $x\in G\setminus \Sol(G)$. Then $\mathcal{NS}_G$ is hamiltonian.
\end{thm}
\begin{proof}
Note that  $\deg(x) = |G| - |\Sol_G(x)|$ for all $x \in G \setminus \Sol(G)$.
 Since $|\Sol_G(x)|\leq \frac{|G|}{2}$ for all $x \in G \setminus \Sol(G)$ we have $|G| \geq 2|\Sol_G(x)|$. Thus, it follows that $\deg(x) > (|G|-|\Sol(G)|)/2$. Therefore by Dirac's Theorem \cite[p. 54]{bon}, $\mathcal{NS}_G$ is hamiltonian.
\end{proof}

\begin{cor}
 The non-solvable graph of the group $PSL(3,2) \rtimes \mathbb{Z}_2$, $A_6$ and $PSL(2,8)$ are hamiltonian. 
\end{cor}

\begin{proof}
The result follows from Theorem \ref{hamiltonian-3} using the fact that   $|\Sol_G(x)|\leq \frac{|G|}{2}$ for all $x \in G\setminus \Sol(G)$ where $G = PSL(3,2) \rtimes \mathbb{Z}_2, A_6$ and $PSL(2,8)$.
\end{proof}

The following result shows that there is a group $G$   with  $|\Sol_G(x)|> |G|/2$ for some $x \in G\setminus \Sol(G)$ such that $\mathcal{NS}_G$ is hamiltonian.

\begin{prop} 
The non-solvable graph of $A_5$ is hamiltonian.
\end{prop}
\begin{proof}
For any two vertex $a$ and $b$ we write $a \sim b$ if  $a$ is adjacent to $b$. It can be  verified that

 $(1,5,4,3,2)\sim(1,3)(2,5)\sim(2,3,4)\sim(1,4)(3,5)\sim(2,5,4)\sim(1,2)(3,4)\sim(1,5,4)\sim(2,5)(3,4)
\sim(1,3,5)\sim(1,4)(2,5)\sim(2,4,3)\sim(1,3)(4,5)\sim(1,2,5)\sim(1,4)(2,3)\sim(3,5,4)\sim(1,5)(2,4)\sim(1,2,3)\sim(1,5)(3,4)\sim(2,3,5)\sim(1,4,2)\sim(2,3)(4,5)\sim(1,5,2)\sim(2,4)(3,5)\sim(1,4,5)\sim(1,2)(3,5)\sim(1,3,4)\sim(1,2)(4,5)\sim(1,5,3)\sim(1,4,2,5,3)\sim(1,3,2)\sim(3,4,5)\sim(1,3)(2,4)\sim(2,5,3)\sim(1,2,4)\sim(1,5)(2,3)\sim(2,4,5)\sim(1,4,3)\sim(1,3,5,2,4)\sim(1,4,5,3,2)\sim(1,2,3,4,5)\sim(1,2,4,3,5)
\sim(1,5,3,2,4)\sim(1,4,5,2,3)\sim(1,5,4,2,3)\sim(1,3,4,5,2)\sim(1,5,3,4,2)$ $\sim
(1,3,2,4,5) \, \sim \, (1,3,2,5,4) \, \sim \, (1,2,4,5,3) \, \sim \, (1,2,5,4,3) \, \sim \, (1,5,2,3,4)\sim(1,2,3,5,4)\sim(1,4,3,2,5) \sim(1,4,3,5,2)\sim(1,3,4,2,5)\sim(1,4,2,3,5)\sim(1,5,2,4,3)$ $\sim(1,3,5,4,2)\sim(1,2,5,3,4)\sim(1,5,4,3,2)$ 

\noindent is a hamiltonian cycle of $\mathcal{NS}_{A_5}$. Hence, $\mathcal{NS}_{A_5}$ is hamiltonian.
\end{proof}

We conclude this section with the following question.

\begin{ques}
 Is  $\mathcal{NS}_G$  hamiltonian for any finite non-solvable group $G$?
\end{ques}

\section{Domination number and vertex connectivity}
 For a graph $\Gamma$ and a subset $S$ of the vertex set $V(\Gamma)$ we write $N_\Gamma[S] = S \cup (\cup_{x \in S}\nbd(x))$.
 If $N_\Gamma[S] = V(\Gamma)$ then $S$ is said to be a \textit{dominating set} of $\Gamma$.  The \textit{domination number} of  $\Gamma$, denoted by $\lambda(\Gamma)$, is the minimum cardinality of  dominating sets of  $\Gamma$. In this section, we shall obtain a few results regarding $\lambda(\mathcal{NS}_G)$.

\begin{prop}\label{domi1}
 Let $G$ be a finite non-solvable group. Then $\lambda(\mathcal{NS}_G) \neq 1$.
\end{prop}
\begin{proof} 
Let $\{x\}$ be a dominating set for $\mathcal{NS}_G$. If $\Sol(G)$ contains a non-trivial element $z$ then $xz$ is  adjacent to $x$, a contradiction. Hence,  $|\Sol(G)| = 1$. 

If $o(x) \ne 2$ then $x$ is   adjacent to $x^{-1}$, which is a contradiction. Hence, 
$o(x) = 2$ and so $x\in P_2$, for some Sylow $2$-subgroup $P_2$ of $G$. Since $|\Sol(G)| = 1$ and $x$ is adjacent to all vertices of $\mathcal{NS}_G$ we have $\Sol_G(x) =\langle x\rangle$. Also, $P_2\subseteq \Sol_G(x)$ and so $P_2=\langle x\rangle$. If $Q_{2}$ is another Sylow $2$-subgroup of $G$ then $|Q_2|=2$ and so $\langle P_2, Q_2\rangle$ is a dihedral group and hence solvable. That is, $x$ is not adjacent to $y\in Q_2$, $y\neq 1$, which is a contradiction. Thus it follows that $P_2$ is normal in $G$. Let $g \in G\setminus P_2$. Then $gxg^{-1}=x$, that is $xg=gx$ and so $x\in Z(G)$, which is a contradiction. Hence, the result follows.
\end{proof}

Using GAP \cite{gap} it can be seen that $\lambda(A_5) = \lambda(S_5) = 4$. In fact $\{(3,4,5), (1,2,3,4,5)$, $(1,2,4,5,3), (1,5)(2,4)\}$ and $\{(4,5), (1,2)(3,4,5), (1,3)(2,4,5), (1,5)(2,4)\}$ are dominating sets for $A_5$ and $S_5$ respectively.
%
%
%
At this point we would like to ask the following question.

\begin{ques}
 Is there any finite non-solvable group $G$ such that $\lambda(\mathcal{NS}_G)=2,3$?
\end{ques}

\begin{prop} 
Let $G$ be a non-solvable group. Then a subset $S$ of $V(\mathcal{NS}_G)$ is a dominating set if and only if $\Sol_G(S) \subset \Sol(G)\cup S$.
\end{prop}
\begin{proof}
Suppose $S$ is a dominating set. If $a \not \in \Sol(G)\cup S$ then, by definition of dominating set, there exists $x \in S$  such that $\langle x,a \rangle$ is not solvable. Thus $a \not \in \Sol_G(S)$. It follows that $\Sol_G(S) \subset S \cup \Sol(G)$. 

Now assume that $\Sol_G(S) \subset \Sol(G) \cup S$. If $a \not \in \Sol(G) \cup S$, then by hypothesis, $a \not \in \Sol_G(S)$. Therefore, $a$ is adjacent to at least one element of $S$. This completes the proof.
\end{proof}

The vertex connectivity of a connected graph $\Gamma$, denoted by $\kappa(\Gamma )$,  is defined as the smallest number of vertices whose removal makes the graph disconnected. A subset $S$ of the vertices of a connected graph $\Gamma$ is called a vertex cut set, if $\Gamma \setminus S$ is not  connected but $\Gamma \setminus H$ is connected for any proper subset $H$ of $S$.  We conclude this section with the following result on vertex cut set and vertex connectivity of $\mathcal{NS}_G$.

\begin{prop} 
Let $G$ be a finite non-solvable group and let $S$ be a vertex cut set of $\mathcal{NS}_G$. Then $S$ is a union of cosets of $\Sol(G)$. In particular $\kappa(\mathcal{NS}_G)=t|\Sol(G)|$, where $t > 1$ is an integer.
\end{prop}
\begin{proof}
 Let $a \in S$. Then there exist two distinct components $G_1$, $G_2$ of $\mathcal{NS}_G \setminus S$ and two vertices $x \in V(G_1)$, $y \in V(G_2)$ such that $a$ is adjacent to both $x$ and $y$. By Lemma \ref{sol}, $x$ and $y$ are also adjacent to $az$ for any $z \in \Sol(G)$, and so $a\Sol(G) \subset S$. Thus $S$ is a union of cosets of $\Sol(G)$. Hence, $\kappa(\mathcal{NS}_G)= t|\Sol(G)|$, where $t \geq 1$ is an integer.

Suppose that $|S| = \kappa(\mathcal{NS}_G)$. It follows from the first part that $\kappa(\mathcal{NS}_G) = t |\Sol(G)|$ for some integer $t \geq 1$. If $t = 1$ then $S = b\Sol(G)$ for some element $b \in G\setminus \Sol(G)$. Therefore, there exist two distinct components $G_1$, $G_2$ of $\mathcal{NS}_G \setminus S$ and $r\in V(G_1)$,  $s\in V(G_2)$ such that $b$ is adjacent to both $r$ and $s$. In other words, $\langle b, r\rangle$ and  $\langle b,s\rangle$ are not solvable.
Suppose that $o(b) \ne 2$. Then the number of integers less than $o(b)$ and relatively prime to it is greater  or equal to $2$.  
Let $1 \ne i\in \mathbb{N}$ such that $\gcd(i, o(b)) = 1$. Then 
\[
\langle b^i, r\rangle = \langle b,r\rangle \text{ and } \langle b^i, s\rangle=\langle b,s\rangle. 
\]
Therefore, $b^i$ is adjacent to both $r$ and $s$. This is a contradiction since $b^i \notin b\Sol(G)$. Hence, $o(b) = 2$.

Suppose $x' \in V(G_1)$ and $y' \in V(G_2)$ are adjacent to $bz$ for some $z \in \Sol(G)$. Then, by Lemma \ref{sol} (c),   $b$ is adjacent to $x'$ and $y'$. Again, by Lemma \ref{solo},  $x'\Sol(G)$ and $y'\Sol(G)$ are adjacent to $b\Sol(G)$ in the graph $\mathcal{NS}_{G/\Sol(G)}$. That is, $g\Sol(G)$ and $b\Sol(G)$ are adjacent for all $g\Sol(G) \in V(\mathcal{NS}_{G/\Sol(G)})$. Therefore, $\{b\Sol(G)\}$ is a dominating set of $\mathcal{NS}_{G/\Sol(G)}$ and so $\lambda(\mathcal{NS}_{G/\Sol(G)}) = 1$. Hence, the result follows in view of Proposition \ref{domi1}.
\end{proof}

%
%

\section{Independence Number}

 A subset $X$ of the vertices of a graph $\Gamma$ is called an \textit{independent set} if the induced subgraph on $X$ has no edges. The maximum size of an independent set in a graph $\Gamma$ is called the \textit{independence number} of  $\Gamma$ and it is denoted by $\alpha(\Gamma)$. In this section we consider the following question.  
\begin{ques}\label{indno-1} 
Suppose $G$ is a  non-solvable group such that $\mathcal{NS}_G$  has no infinite independent set. Is it true that 
$\alpha(\mathcal{NS}_G)$ is finite?
\end{ques}
It is worth mentioning that Question \ref{indno-1} is  similar to \cite[Question 2.10]{Ab06} and \cite[Question 3.17]{ns17} where Abdollahi et al. and Nongsiang et al. considered non-commuting and non-nilpotent graphs of finite groups respectively. Note that the group considered in \cite{ns17} in order to answer \cite[Question 3.17]{ns17} negatively, also gives negative answer to Question \ref{indno-1}. However, the next theorem gives affirmative answer to Question \ref{indno-1} for some classes of groups.





\begin{thm}\label{indno-2} 
Let $G$ be a non-solvable group such that $\mathcal{NS}_G$ has no infinite independent sets. If $\Sol(G)$ is a subgroup and $G$ is an Engel, locally finite, locally solvable, a linear group or a $2$-group then $G$ is a finite group. In particular $\alpha(\mathcal{NS}_G)$ is finite. 
\end{thm}
Note that Theorem \ref{indno-2} and its proof are similar to \cite[Theorem 2.11]{Ab06} and \cite[Theorem 3.18]{ns17}.

\begin{prop}\label{inde}
Let $G$ be a group. Then for every maximal independent set $S$ of $\mathcal{NS}_G$ we have 
$$\underset{x\in S}\cap \Sol_G(x)=S\cup \Sol(G).$$
\end{prop}
\begin{proof} 
The result follows from the fact that $S$ is maximal and $S\cup \Sol(G)\subset \Sol_G(x)$ for all $x \in S$. 
\end{proof}

\begin{rem} 
 Let $R=\{(3,4,5),(1,4)(3,5),(2,5,3)\}\subset A_5$. Then $R$ is an independent set of $\mathcal{NS}_{A_5}$  and $\langle R\rangle \cong A_5$. This shows that a subgroup generated by an independent set may not be solvable. Also there exist a maximal independent set $S$, such that $R\subseteq S$. Since the edge set of $\mathcal{NS}_{A_5}$ is non-empty, we have $S\neq A_5\setminus \Sol(A_5)$, showing that $S\cup \Sol(A_5)$ is not a subgroup of $A_5$. Thus for a finite non-solvable group $G$ and a maximal independent set $S$ of $\mathcal{NS}_G$, $S\cup \Sol(G)$ need not be a subgroup of $G$.
\end{rem}

We conclude this section with the following result.
\begin{thm}\label{fingen}
The order of a finite non-solvable group $G$ is bounded by a function of the independence number of its non-solvable graph. Consequently, given a non-negative integer $k$,  there are at the most finitely many finite non-solvable groups whose non-solvable graphs  have independence number $k$.
\end{thm}
\begin{proof} 
Let $x \in G$ such that $x,x^2 \not \in \Sol(G)$. Then $x\Sol(G)\cup x^2\Sol(G)$ is an independent set of $\mathcal{NS}_G$. Thus $|\Sol(G)|\leq \frac{k}{2}$. Let $P$ be a Sylow subgroup of $G$, then $P$ is solvable. Thus it follows that $P\setminus \Sol(G)$ is an independent set of $G$. Hence $|P\setminus \Sol(G)|\leq k$, that is $|P|\leq \frac{3k}{2}$. Since, the number of primes less  or equal to $\frac{3k}{2}$ is at most $\frac{3k}{4}$, we have $|G|\leq (\frac{3k}{2})^{\frac{3k}{4}}$. This completes the proof.
\end{proof}








\section{Clique number of non-solvable graphs}\label{S:cli}

A subset of the vertex set of a graph $\Gamma$ is called a \textit{clique} of $\Gamma$ if it consists entirely of pairwise adjacent vertices. The least upper bound of the sizes of all the cliques of $\Gamma$ is called the \textit{clique number} of $\Gamma$, and is denoted by $\omega (\Gamma)$. Note that $\omega (\tilde{\Gamma}) \leq \omega (\Gamma)$ for any subgraph $\tilde{\Gamma}$ of $\Gamma$.

\begin{prop} 
Let $G$ be a finite non-solvable group.
\begin{enumerate}
\item [\rm(a)] If $H$ is a non-solvable subgroup of $G$ then $\omega(\mathcal{NS}_H)\leq \omega(\mathcal{NS}_G)$.
\item [\rm(b)] If $\frac{G}{N}$ is non-solvable then $\omega(\mathcal{NS}_{\frac{G}{N}}) \leq \omega(\mathcal{NS}_G)$. The equality holds when $N = \Sol(G)$.
\end{enumerate}	
\end{prop}
\begin{proof}
Part (a) follows from the fact that $\mathcal{NS}_H$ is a subgraph of $\mathcal{NS}_G$.
For part (b), we shall show that $\mathcal{NS}_{\frac{G}{N}}$ is isomorphic to a subgraph of $\mathcal{NS}_G$.

 Let $V(\mathcal{NS}_{\frac{G}{N}})=\{x_1N,x_2N,\dots,x_nN\}$ and $K=\{x_1,x_2,\dots,x_n\}$. Then, for $x_iN \in V(\mathcal{NS}_{\frac{G}{N}})$,  there exist $x_jN \in V(\mathcal{NS}_{\frac{G}{N}})$  such that $\langle x_iN,x_jN\rangle$ is not solvable.
  Let $H = \langle x_i, x_j\rangle$. Then
\[
 \langle x_iN, x_jN \rangle = \frac{H N}{N}.
\]
%
Suppose $H$ is solvable. Then there exists a sub-normal series $\{1\} = H_0 \leq H_1 \leq H_2 \leq \cdots \leq H_n = H$, where $H_i$ is normal in $H_{i+1}$ and $\frac{H_{i+1}}{H_i}$ is abelian for all $i = 0, 1, \dots, n - 1$. Consider the series $N = H_0N \leq H_1N \leq \cdots \leq H_nN = HN$. We have $H_iN$ is normal in $H_{i+1}N$, for if $an\in H_iN$ and $bm\in H_{i+1}N$ then $bman(bm)^{-1} \in H_iN$. Also, $\frac{H_{i+1}N}{H_iN}$ is abelian, for if $a,b \in \frac{H_{i+1}N}{H_iN}$ then $a = kn_1(H_iN) = k(H_iN)$ and $b=ln_2(H_iN) = l(H_iN)$. Therefore, $ab=kl(H_iN)$. Since $H_{i+1}/H_i$ is abelian, we have $kH_ilH_i = lH_ikH_i$, that is $klH_i = lkH_i$. Thus 
\[
ab = kl(H_iN) = (klH_i)N = (lkH_i)N = lk(H_iN) = ba.
\]
Therefore, $\frac{H_{i+1}N}{H_iN}$ is abelian. Hence, $HN$ is solvable and so $HN/N = \langle x_iN,x_jN \rangle$ is also solvable; which is a contradiction. Therefore, $H$ is non-solvable.
Let $L$ be a graph such that  $V(L)= K$ and two vertices $x, y$ of $L$ are adjacent if and only if $xN$ and $yN$ are adjacent in $\mathcal{NS}_{\frac{G}{N}}$. Then $L$ is a subgraph of $\mathcal{NS}_G[K]$ and hence a subgraph of $\mathcal{NS}_G$. 
Define a map $\phi: V(\mathcal{NS}_{\frac{G}{N}}) \to V(L)$ by $\phi(x_iN)=x_i$. Then $\phi$ is one-one and onto. Also two vertices $x_iN$ and $x_jN$ are adjacent in $\mathcal{NS}_{\frac{G}{N}}$ if and only if $x_i$ and $x_j$ are adjacent in $L$. Thus $\mathcal{NS}_{\frac{G}{N}} \cong L$. 

If $N = \Sol(G)$ then, by Lemma \ref{solo}, it follows that $\{x_1\Sol(G), x_2\Sol(G),\dots$, $x_t\Sol(G)\}$ is a clique of $\mathcal{NS}_\frac{G}{\Sol(G)}$ if and only if $\{x_1, x_2, \dots, x_t\}$ is a clique of $\mathcal{NS}_G$. Hence,  $\omega(\mathcal{NS}_\frac{G}{\Sol(G)})= \omega(\mathcal{NS}_G)$.
\end{proof}


\begin{thm} 
For any non-solvable group $G$ and a solvable group $S$ we have 
\[
\omega(\mathcal{NS}_G)=\omega(\mathcal{NS}_{G\times S}).
\]
\end{thm}
\begin{proof}
Suppose $C$ is a clique of $\mathcal{NS}_G$. Let $a,b \in C$, then $\langle a,b \rangle$ is not solvable. Now, $\langle (a,e_s),(b,e_s)\rangle \cong \langle a,b \rangle$, where $e_s$ is the identity element of $S$, and so $\langle (a,e_s),(b,e_s)\rangle$ is not solvable. Thus $C\times \{e_s\}$ is a clique of $\mathcal{NS}_{G\times S}$. Now suppose $D$ is a clique of $\mathcal{NS}_{G\times S}$. Let $(x,s_1),(y,s_2)\in D$, where $x \ne y$. Then $\langle (x,s_1),(y,s_2)\rangle \subseteq \langle x,y\rangle \times \langle s_1,s_2\rangle$. Since $S$ is solvable, we have $\langle s_1,s_2 \rangle$ is solvable. Since $\langle (x,s_1),(y,s_2)\rangle$ is not solvable, we have $\langle x,y \rangle$ is not solvable. Thus  $E=\{x : (x,s)\in D\}$ is a clique of $\mathcal{NS}_G$. Hence, the result follows noting that  $|D|=|E|$. 
\end{proof}
The following lemma is useful in obtaining a lower bound for $\omega(\mathcal{NS}_G)$.
\begin{lem} \label{order-lem}
 Let $G$ be a finite non-solvable group. Then there exists an element $x \in G\setminus \Sol(G)$ such that $o(x)$ is a prime greater or equal to $5$.
\end{lem}
\begin{proof}
Suppose that $1 \ne o(x) = 2^{\alpha}3^{\beta}$ for all $x \in G\setminus \Sol(G)$, where $\alpha$ and $\beta$ are non-zero integers. Then $|G/\Sol(G)| = 2^{m}3^{n}$ for some non-zero integers $m, n$. Therefore,  $G/\Sol(G)$ is solvable and so, by Lemma \ref{solo}, $G$ is solvable; a contradiction. This proves the existence of an element $x \in G\setminus \Sol(G)$ such that $o(x)$ is a prime greater or equal to $5$.  
\end{proof}
\begin{prop}\label{geq6}
 Let $G$ be a finite non-solvable group. Then $\omega(\mathcal{NS}_G)\geq 6$.
\end{prop}
\begin{proof}
By Lemma \ref{order-lem}, we have an element $x \in G\setminus \Sol(G)$ such that $o(x)$ is a prime greater  or equal to $5$. Let $y \in G\setminus \Sol(G)$ such that $x$ is adjacent to $y$. Then $\{x, y, xy, x^2y, x^3y, x^4y\}$ is a clique of $\mathcal{NS}_G$ and so $\omega(\mathcal{NS}_G)\geq 6$. 
\end{proof}

The following program in GAP \cite{gap} shows that 
\[
\omega(\mathcal{NS}_{A_5}) = \omega(\mathcal{NS}_{SL(2,5)}) =  \omega(\mathcal{NS}_{\mathbb{Z}_2 \times A_5}) = 8 \text{ and } \omega(\mathcal{NS}_{S_5}) = 16.
\]
Note that $A_5 \, = \, \text{SmallGroup}(60, 5), \,\quad  SL(2,5) \, = \, \text{SmallGroup}(120, 5), \,\quad S_5=  $\\ $\text{SmallGroup}(120, 34)$ and $\mathbb{Z}_2 \times A_5 = \text{SmallGroup}(120, 35)$. Also $G/\Sol(G) \cong A_5$ for $G = A_5, SL(2,5)$ and $\mathbb{Z}_2 \times A_5$.


\begin{verbatim}
LoadPackage("GRAPE");
sol:=[60,120];
for n in sol do
 allg:=AllSmallGroups(n);
  for g in allg do
    if IsSolvable(g)=false then
    h:=Graph(g,Difference(g,RadicalGroup(g)),
                           OnPoints,function(x,y) return 
    IsSolvable(Subgroup(g,[x,y]))=false; end, true);
    k:=CompleteSubgraphs(h);
    cn:=[];
    for i in  k 
      do
      AddSet(cn,Size(i));
      od;
      Print("\n",IdGroup(g),",  ",StructureDescription(g),",
                                               cliquenumber=",
      Maximum(cn),"\n");
     fi;
    od;
    n:=n+1;
od;

\end{verbatim}

%
%
%

The following program in GAP \cite{gap} shows that the clique number of $\mathcal{NS}_G$ for groups of order   less or equal to $360$ with $G/\Sol(G) \not \cong A_5$ is greater  or equal to $9$.
\begin{verbatim}
n:=120;
while n<=504 do
	allg:=AllSmallGroups(n);
	for g in allg do
		if IsSolvable(g)=false then
			rad:=RadicalGroup(g);
			m:=Size(rad);
			l:=n/m;
			if l>60 then
				dif:=Difference(g,rad);
				for x in dif do
					clique:=[x];
					p:=0;
					for y in dif do
						i:=0;
						for z in clique do
							if IsSolvable(Subgroup(g,[y,z]))=true then
								i:=1;
								break;
							fi;
						od;
						if i=0 then
							AddSet(clique,y);
						fi;			
						if Size(clique)>9 then
							p:=1;
							break;
						fi;	
					od;
					if p=1 then
						break;
					fi;

				od;
				if p=1 then
					Print(IdGroup(g), "Clique greater than 8", "\n","\n");
				else
					Print(IdGroup(g), "Not Clique greater than 8", "\n","\n");
				fi;
			fi; 

		fi;
	od;
	n:=n+1;
od;

\end{verbatim}


We conclude this section with the following conjecture.
\begin{conj}
 Let $G$ be a non-solvable group such that $\omega(\mathcal{NS}_G) = 8$. Then $G/\Sol(G)\cong A_5$.
\end{conj}

\section{Groups with the same Non-solvable graphs}\label{S:sam}


In \cite{Ab06}, Abdollahi  et al. conjectured that if $G$ and $H$ are two non-abelian finite groups such that their non-commuting graphs are isomorphic then $|G| = |H|$. They also verified this conjecture for several classes of finite groups. Recently, Nongsiang and Saikia \cite{ns17} posed similar conjecture for non-nilpotent graphs of finite groups. In this section we consider the following problem. 

\begin{prob}
Let $G$ and $H$ be two non-solvable groups such that $\mathcal{NS}_G \cong \mathcal{NS}_H$. Determine whether  $|G| = |H|$.
\end{prob}

We begin the section with the following theorem.

\begin{thm}\label{GfiHfi} 
Let $G$ and $H$  be two non-solvable groups such that $\mathcal{NS}_G \cong \mathcal{NS}_H$. If $G$ is finite then $H$ is also  finite. Moreover, $|\Sol(H)|$ divides 
\[
\gcd(|G|-|\Sol(G)|,|G|-|\Sol_G(x)|,|\Sol_G(g)|-|\Sol(G)|),
\]
where  $g \in G\setminus \Sol(G)$, and hence $|H|$ is bounded by a function of $G$.
\end{thm}
\begin{proof} 
Since $\mathcal{NS}_G \cong \mathcal{NS}_H$, we have $|H\setminus \Sol(H)| = |G\setminus \Sol(G)|$ and so $|H\setminus \Sol(H)|$ is finite. If $h \in H\setminus \Sol(H)$ then $\{aha^{-1} : a \in H\} \subset H\setminus \Sol(H)$, since $\Sol(H)$ is closed under conjugation. Thus every element in $H\setminus \Sol(H)$ has finitely many conjugates in $H$. It follows that $K = C_H (H\setminus \Sol(H))$ has finite index in $H$. Since $\mathcal{NS}_H$ has no isolated vertex,  there exist two adjacent vertices $u$ and $v$ in $\mathcal{NS}_H$. Now, if $s \in K$ then $s \in C_H (\{u,v\})$ and so $\langle su,v \rangle$ is not solvable since $\langle su, v \rangle \cong \langle u, v \rangle \times \langle s \rangle$. Therefore $Ku \subset H\setminus \Sol(H)$ and so $K$ is finite. Hence, $H$ is  finite.

It follows that $\Sol(H)$ is a subgroup of $H$ and so $|\Sol(H)|$ divides $|H|-|\Sol(H)|$. Since $|H|-|\Sol(H)| =|G| - |\Sol(G)|$, we have $|\Sol(H)|$ divides $|G| - |\Sol(G)|$. Let $x' \in H \setminus \Sol(H)$ and $ y \in \Sol_H(x')$. Then, by Lemma \ref{sol}, $\langle x', yz \rangle$ is solvable for all $z \in \Sol(H)$. Thus $\Sol_H(x') = \Sol(H)\cup y_1\Sol(H) \cup \dots \cup y_n \Sol(H)$, for some $y_i \in H$. Therefore $|\Sol(H)|$ divides $|\Sol_H(x')|$ and so $|\Sol(H)|$ divides $|H|-|\Sol_H(x')|$.
We have $\deg(\mathcal{NS}_G) = \deg(\mathcal{NS}_H)$ since $\mathcal{NS}_G \cong \mathcal{NS}_H$. Also  
 $\deg(g) = |G| - |\Sol_G(g)|$ for any $g \in V(\mathcal{NS}_G)$ and $\deg(h) = |H| - |\Sol_H(h)|$ for any $h \in V(\mathcal{NS}_H)$. Therefore $|\Sol(H)|$ divides $|H| - |\Sol_H(h)|$ and hence  $|G| - |\Sol_G(g)|$ for any $g \in G \setminus \Sol(G)$. Since  $|\Sol(H)|$ divides $|G| - |\Sol(G)|$ and $|G| - |\Sol_G(g)|$, it divides $|G| - |\Sol(G)| - (|G| - |\Sol_G(g)|)= |\Sol_G(g)| - |\Sol(G)| $. This completes the proof.
\end{proof}

\begin{prop} Let $G$ be a non-solvable group such that $\mathcal{NS}_G$ is finite. Then $G$ is a finite group.
\end{prop}
\begin{proof} It follows directly from the first paragraph of the proof of the above theorem. 
\end{proof}

\begin{prop} Let $G$ be a group such that $\mathcal{NS}_G\cong \mathcal{NS}_{A_5}$, then $G\cong A_5$.
\end{prop}
\begin{proof}
Since  $\mathcal{NS}_G\cong \mathcal{NS}_{A_5}$, we have $G$ is a finite non-solvable group and 
\[
|G\setminus \Sol(G)|=|A_5\setminus \Sol(A_5)|=59.
\]
Therefore, $|G|=|\Sol(G)| + 59$. Since $\Sol(G)$ is a subgroup of $G$, we have $|\Sol(G)|\leq \frac{|G|}{2}$ and so $|G|\leq 118$. 
Hence, the result follows.
\end{proof}

\begin{rem}
  Using the following program in GAP \cite{gap}, one can see that the non-solvable graphs of $SL(2,5)$ and $\mathbb{Z}_2\times A_5$ are isomorphic. It follows that  non-solvable graphs  of two groups are isomorphic need not implies that their corresponding groups are isomorphic.
\begin{verbatim}

LoadPackage("GRAPE");
g:=SmallGroup(120,5);
solg:=RadicalGroup(g);
gmc:= Difference(g,solg);
m:=Size(gmc);
h:=SmallGroup(120,35);
hmc:=Difference(h,RadicalGroup(h));
if m=Size(hmc) then
   gg:=Graph(g,gmc,OnPoints,function(x,y) return
                               IsSolvable(Subgroup(g,[x,y]))=
   false; end, true);
   gh:=Graph(h,hmc,OnPoints,function(x,y) return 
                               IsSolvable(Subgroup(h,[x,y]))=
   false; end, true);
   if IsIsomorphicGraph(gg,gh)=true then
     Print("\n","\n","an example of G and H isomorphic
                                  but not of same order.doc","G= ",
     StructureDescription(g), ", ", "	Id=", IdGroup(g)," H = ",
     StructureDescription(h)," Id=", IdGroup(h),"\n","\n");
   fi;
fi;
\end{verbatim}

\end{rem}

\begin{prop} 
Let $G$ and $H$ be two finite non-solvable groups. If $\mathcal{NS}_G \cong \mathcal{NS}_H$ then $\mathcal{NS}_{G \times A} \cong \mathcal{NS}_{H \times B}$, where $A$ and $B$  are two solvable groups  having  equal order.  
\end{prop}
\begin{proof}
Let $\varphi :\mathcal{NS}_G \to \mathcal{NS}_H$ be a graph isomorphism and $\psi :A \to B$ be a bijective map. Then  $(g, a) \mapsto (\varphi(g),\psi(a))$  defines  a graph isomorphism between $\mathcal{NS}_{G \times A}$ and $\mathcal{NS}_{H \times B}$. 
\end{proof}

A non-solvable group $G$ is called an $Fs$-group if for every two elements $x , y \in G\setminus \Sol(G)$ such that $\Sol_G(x)\neq \Sol_G(y)$ implies $\Sol_G(x) \not \subset \Sol_G(y)$ and $\Sol_G(y) \not \subset \Sol_G(x)$.

\begin{prop} 
Let $G$ be an $Fs$-group. If $H$ is a non-solvable group such that $\mathcal{NS}_G \cong \mathcal{NS}_H$  then $H$ is also an  $Fs$-group.
\end{prop}
\begin{proof}
Let $\psi: \mathcal{NS}_H \to  \mathcal{NS}_G$ be a graph isomorphism.  Let $x, y \in H\setminus \Sol(H)$ such that $\Sol_H(x) \subseteq \Sol_H(y)$. Then $\psi(\Sol_H(x)\setminus \Sol(H)) \subseteq \psi(\Sol_H(y)\setminus \Sol(H))$. We have
\[
\psi(\Sol_H(z)\setminus \Sol(H)) = \Sol_G(\psi(z))\setminus \Sol(G) \text{ for all } z\in H\setminus \Sol(H). 
\]
Therefore, $\Sol_G(\psi(x))\setminus \Sol(G) \subseteq \Sol_G(\psi(y))\setminus \Sol(G)$. 
Since $G$ is an $Fs$-group, we have 
\[
\Sol_G(\psi(x))\setminus \Sol(G) = \Sol_G(\psi(y))\setminus \Sol(G).
\]
It follows that $\Sol_H(x)\setminus \Sol(H) =  \Sol_H(y)\setminus \Sol(H)$ and so $\Sol_H(x) = \Sol_H(y)$. Hence, $H$ is an $Fs$-group.
\end{proof}

\section{Genus of non-solvable graph}\label{S:gen}
The \textit{genus} of a graph $\Gamma$, denoted by $\gamma(\Gamma)$, is the smallest non-negative integer $g$ such that the graph can be embedded on the surface obtained by attaching $g$ handles to a sphere. Clearly, if $\tilde{\Gamma}$ is a subgraph of $\Gamma$ then $\gamma(\tilde{\Gamma}) \leq \gamma(\Gamma)$. Graphs having genus zero are called \textit{planar} graphs, while those having genus one are called \textit{toroidal} graphs. Graphs having genus two are called {\em double-toroidal} graphs and those having genus three are called {\em triple-toroidal} graph. In \cite[Corollary 3.14]{hr}, it was shown that $\mathcal{NS}_G$ is not planar  for finite non-solvable group   $G$.  
In this section, we extent \cite[Corollary 3.14]{hr} and  show that $\mathcal{NS}_G$ is neither planar,  toroidal,  double-toroidal nor triple-toroidal.



It is well-known (see \cite[Theorem 6-39]{whi}) that
$\gamma(K_{n})=\left\lceil \frac{(n-3)(n-4)}{12}\right 
\rceil$,
where $n\geq 3$ and $K_n$ is the complete graph on $n$ vertices. Also, if $m,n\geq 2$ then
\[\gamma(K_{m,n})=\left\lceil \frac{(m-2)(n-2)}{4}
\right\rceil \text{ and }  \gamma(K_{m, m, m})= \frac{(m - 2)(m - 1)}{2},\]
where  $K_{m,n}$, $K_{m,m, m}$ are complete  bipartite and tripartite graphs respectively.


\begin{prop}\label{finsol}
 Let $G$ be a finite non-solvable group.  Then 
\[
 |\Sol(G)|\leq \sqrt{2\gamma(\mathcal{NS}_G)} + 2.
\]
\end{prop}
\begin{proof} 
Assume that $Z=\Sol(G)$. By Proposition \ref{geq6}, we have $\omega(\mathcal{NS}_G) \geq 3$. So, there exist $u,v,w \in G\setminus Z$ such that they are adjacent to each other. 
Then, by Lemma \ref{sol}, $\mathcal{NS}_G[uZ\cup vZ \cup wZ]$ is isomorphic to $K_{|Z|,|Z|,|Z|}$. 
We have 
\[
\gamma(\mathcal{NS}_G) \geq \gamma( K_{|Z|,|Z|,|Z|})=\frac{(|Z|-2)(|Z|-1)}{2} \geq \frac{(|Z|-2)(|Z|-2)}{2} 
\]
and hence the result follows.
\end{proof}
\begin{thm}\label{nopla} 
Let $G$ be a finite non-solvable graph. Then $\gamma(\mathcal{NS}_G)\geq 4$. In particular, $\mathcal{NS}_G$ is neither  planar,  toroidal,  double-toroidal nor triple-toroidal.
\end{thm}
\begin{proof}
%
By Lemma \ref{order-lem}, we have an element $x \in G\setminus \Sol(G)$ such that $o(x)$ is a prime greater  or equal to $5$.
Clearly, $\nbd(x)\neq \emptyset$. Assume that $o(y) = 2$ for all 
 $y\in \nbd(x)$. Then $xy \in \nbd(x)$ and so $o(xy)=2$. 
Thus $\langle x, y\rangle = \langle y, xy\rangle$  is isomorphic to a dihedral group,
 which is  a contradiction. Therefore, there exist $y \in \nbd(x)$  such that $o(y)\geq 3$. Let $1 \ne j \in \mathbb{N}$  and $\gcd(j, o(x)) = 1$.
Consider the subsets $H = \{x, x^2, x^3, x^4\}$,  $K = \{y^ix^j :   i = 1, 2, j= 0, 1, 2, 3, 4\}$ of $G\setminus \Sol(G)$ and the induced graph  $\mathcal{NS}_G[H\cup K]$. Notice that    $\mathcal{NS}_G[H\cup K]$ has a subgraph isomorphic to $K_{4,10}$ and hence   
\[
\gamma(\mathcal{NS}_G) \geq \gamma(\mathcal{NS}_G[H\cup K]) \geq \gamma(K_{4,10}) = 4.
\]
 This completes the proof. 
\end{proof}

\begin{rem}\label{gengr12}
 By GAP \cite{gap}, using the following program, we see that $\mathcal{NS}_{A_5}$ has 1140 edges and 59 vertices. Thus by \cite[Corollary 6-14]{whi}, we have $\gamma(\mathcal{NS}_{A_5})\geq \frac{1140}{6}-\frac{59}{2}+1=161.5$ and so $\gamma(\mathcal{NS}_{A_5}) \geq 162$.

\begin{verbatim}
LoadPackage("GRAPE");
g:=AlternatingGroup(5);
solg:=RadicalGroup(g);
h:=Graph(g,Difference(g,solg),OnPoints,function(x,y) return 
IsSolvable(Subgroup(g,[x,y]))= false; end, true);
k:=Vertices(h);
i:=0;
for x in k do
     i:=i+VertexDegree(h,x);
od;
Print("Number of Edges=",i/2);
\end{verbatim}

Similarly $\mathcal{NS}_{S_5}, \mathcal{NS}_{SL(2,5)}$ and $\mathbb{Z}_2\times A_5$ has 4560 edges and 119 vertices. So their genera are at least 732.
\end{rem}

A compact surface $N_k$ is a connected sum of $k$ projective planes. The number $k$ is called the \textit{crosscap} of $N_k$. A simple graph which can be embedded in $N_k$ but not in $N_{k-1}$, is called a graph of crosscap $k$. The notation $\bar{\gamma}(\Gamma)$ stand for the crosscap of a graph $\Gamma$. It is easy to see that $\bar{\gamma}(\Gamma_0)\leq \bar{\gamma}(\Gamma)$ for all subgraphs $\Gamma_0$ of $\Gamma$. Also, a graph $\Gamma$, such that $\bar{\gamma}(\Gamma)=1$ is called a \textit{projective} graph.
It is shown in \cite{ghw} that $2K_5$ is not projective. Hence, any graph containing a subgraph isomorphic to $2K_5$ is not projective. We conclude this paper with the following result.

\begin{thm}
 Let $G$ be a finite non-solvable group. Then $\mathcal{NS}_G$ is not projective.
\end{thm}
\begin{proof}
As shown in  the proof of Theorem \ref{nopla}, there exist $x, y \in G\setminus \Sol(G)$  such that 
$o(x)$ is a prime greater or equal to $5$, $o(y) \geq 3$ and they are adjacent.
 Let $1 \ne j \in \mathbb{N},  \gcd(j, o(y))=1$. Consider the subsets $H = \{y, xy, x^2y, x^3y, x^4y\}$ and $K = \{y^j, xy^j,$ $x^2y^j, x^3y^j, x^4y^j\}$ of $G\setminus \Sol(G)$. Then $H\cap K = \emptyset$ and $\mathcal{NS}_G[H]\cong \mathcal{NS}_G[K]\cong K_5$. It follows that $\mathcal{NS}_G$ has a subgraph isomorphic to  $2K_5$. Hence, $\mathcal{NS}_G$ is not projective.
\end{proof}


\begin{thebibliography}{33}

\bibitem{Ab06}
A. Abdollahi,   S. Akbari  and H. R. Maimani, {\em Non-commuting graph of a group},  J. Algebra, {\bf 298} (2006),  468--492.

\bibitem{Ab10}
A. Abdollahi, M. Zarrin, {\em Non-nilpotent graph of a group}, Comm. Algebra, {\bf 38} (2010), 4390-4403.

\bibitem{Af15} 
M. Afkhami, D. G. M. Farrokhi and K. Khashyarmanesh, {\em Planar, toroidal, and projective commuting and noncommuting graphs}, Comm. Algebra, {\bf 43} (2015), 2964-2970.


\bibitem{akbari}
B. Akbari, {\em More on the non-solvable graphs and solvabilizers}, preprint available at https://arxiv.org/pdf/1806.01012.pdf. 



\bibitem{bnn19}
P. Bhowal, D. Nongsiang and R. K. Nath, {\em A note on solvable graphs of finite groups}, preprint, available at https://arxiv.org/pdf/1903.01755.pdf.


\bibitem{bon}
 J. A. Bondy, J. S. R. Murty, {\em Graph Theory with Applications}, North-Holland, 1976.
 
\bibitem{Dar14}
 M. R. Darafsheh, H. Bigdely, A. Bahrami and M. D. Monfared, {\em Some results on non-commuting graph of a finite group}, Ital. J. Pure Appl. Math., {\bf 268} (2014), 371-387.

\bibitem{ddn18}
P. Dutta, J. Dutta and R. K. Nath, {\em Laplacian spectrum of non-commuting graphs of finite groups},  Indian J. Pure Appl. Math., {\bf 49} (2018), 205-216.



\bibitem{gkps}
 R. Guralnick, B. Kunyavskii, E. Plotkin and A. Shalev, {\em Thompson-like characterizations of the solvable radical}, J. Algebra, \textbf{300} (2006), 363--375.

\bibitem{ghw}
H. H. Glover, J. P. Huneke and C. S. Wang, {\em 103 graphs that are irreducible for the projective plane}, J. Combinatorial Series B, \textbf{27} (1978), 332--370.

\bibitem{gW2000}
R. Guralnick and J. Wilson, {\em The probability of generating a finite soluble group}, Proc. London Math. Soc.  {\bf 81}(3) (2000), 405--427.


\bibitem{hr} 
D. Hai-Reuven, {\em Non-solvable graph of a finite group and solvabilizers}, arXiv:1307.2924v1, 2013.






\bibitem{ns17}
D. Nongsiang and P. K. Saikia, {\em  On the non-nilpotent graphs of a group}, Int. Electron. J. Algebra, \textbf{22} (2017), 78--96. 

\bibitem{ns2017}
D. Nongsiang and P. K. Saikia, {\em On the non-nilpotent graphs of a group}, Int. Electron. J. Algebra, {\bf 22} (2017), 78-96.
 
\bibitem{talebi}
A. A. Talebi, {\em On the non-commuting graphs of group $D_{2n}$}, Int. J. Algebra, {\bf 20} (2008), 957-961.




 

\bibitem{whi} 
A. T. White, {\em Graphs, Groups and Surfaces}, North-Holland Mathematics Studies, no. 8., American Elsevier Publishing Co., Inc.,  New York, 1973.

\bibitem{wIL2008}
J. S. Wilson, {\em The probability of generating a soluble subgroup of a finite group}, J. London Math. Soc. (2) {\bf 75} (2007), 431-446.


\bibitem{gap} The GAP~Group, \emph{GAP -- Groups, Algorithms, and Programming, Version 4.6.4}, 2013
  \verb+(http://www.gap-system.org)+.

\end{thebibliography}
\end{document}